\documentclass[12pt]{amsart}
\usepackage{amssymb}

\newtheorem{theorem}{Theorem}
\newtheorem{proposition}{Proposition}
\newtheorem{corollary}{Corollary}
\newtheorem{lemma}{Lemma}

\theoremstyle{definition}
\newtheorem{definition}{Definition}    
\newtheorem{remark}{Remark}
\newtheorem{example}{Example}

\renewcommand{\SS}{\mathbb S}
\newcommand{\BB}{\mathbb B}
\newcommand{\CC}{\mathbb C}
\newcommand{\RR}{\mathbb R}
\newcommand{\HH}{\mathcal{H}}
\newcommand{\VV}{\mathcal{V}}
\newcommand{\grad}{\mathrm{grad}}
\newcommand{\spec}{\mathrm{spec}}
\newcommand{\Span}{\mathrm{span}}
\newcommand{\rank}{\mathrm{rank}}
\newcommand{\Div}{\mathrm{div}}
\newcommand{\ov}{\overline}
\newcommand{\abs}[1]{\vert #1\vert}
\newcommand{\norm}[1]{\vert #1\vert} 

\begin{document}  

\title{Eigenvalues of harmonic almost submersions}

\author{E. Loubeau}
\address{D{\'e}partement de Math{\'e}matiques, Universit{\'e} de Bretagne Occidentale, Brest, France}
\email{loubeau@univ-brest.fr}

\author{R. Slobodeanu}
\address{Faculty of Physics, Bucharest University, Romania}
\email{radualexandru.slobodeanu@g.unibuc.ro}

\thanks{The second author benefited from a one-year grant from the Conseil G{\'e}n{\'e}ral du Finist{\`e}re and acknowledges partial support by the CEx
grant no. 2-CEx 06-11-22/ 25.07.2006.}


%\dedicatory{}


\begin{abstract}
Maps between Riemannian manifolds which are submersions on a dense subset, are studied by means of the eigenvalues of the pull-back of the target metrics, the first fundamental form. Expressions for the derivatives of these eigenvalues yield characterizations of harmonicity, totally geodesic maps and biconformal changes of metric preserving harmonicity. A Schwarz lemma for pseudo harmonic morphisms is proved, using the dilatation of the eigenvalues and, in dimension five, a Bochner technique method, involving the Laplacian of the difference of the eigenvalues, gives conditions forcing pseudo harmonic morphisms to be harmonic morphisms.
\end{abstract}

\maketitle

\section{Introduction}

The geometric features of harmonic maps are often best revealed when combined with other properties such as conformality, which has led to a connection with minimal submanifolds and, more generally, minimal branched immersions. The dimensional counterpart of this approach emerged from a seemingly unrelated problem on maps which preserve, by composition on the right-hand side, (local) harmonic functions. 
These maps, called harmonic morphisms, were characterized by Fulgede~\cite{Fuglede} and Ishihara~\cite{Ishihara} as harmonic maps with the added property of  horizontal weak conformality, i.e. the map $\varphi : (M,g) \to (N,h)$ between Riemannian manifolds, is horizontally weakly conformal if at any point $x\in M$, either $d\varphi_{x} \equiv 0$ or $d\varphi_{x}\big|_{(\ker d\varphi_{x})^{\perp}}$ is surjective and conformal. The conformal factor $\lambda(x)$ is called the dilation of $\varphi$ and harmonic morphisms must be almost submersions, i.e. submersive on a dense set. The theory of harmonic morphisms has developed into a rich subject, cf.~\cite{BW}, and one particular interesting area is its interaction with complex structures and links with holomorphic maps.
In general, horizontally weakly conformal maps pull back the metric of the target onto a metric on the horizontal distribution (the orthogonal complement of the kernel of the differential) but in the presence of a complex structure on the codomain, one can also pull back this geometric structure. However, this property is not characteristic of horizontally weakly conformal maps but defines a larger class called pseudo horizontally weakly conformal maps, first identified by Burns, Burstall, de Bartolomeis and Rawnsley in \cite{BBBR} in their study of stable harmonic maps into irreducible Hermitian symmetric spaces of compact type and formulated as
$$[d\varphi \circ (d\varphi)^t , J]=0,$$
where $(d\varphi)^t$ is the adjoint of the $d\varphi$, and $J$ is the complex structure on the target.
A more systematic study of pseudo horizontally weakly conformal maps into almost Hermitian manifolds was carried out in \cite{Loubeau} and \cite{LM} where the pull-back of the complex structure $J$ is used to build an $f$-structure on the domain $(M,g)$, i.e. an endomorphism $F$ of the tangent bundle satisfying
$$ F^3 + F =0,$$
associated to each pseudo horizontally weakly conformal map. Moreover, these maps are $(F,J)$-holomorphic, i.e. their differentials intertwine the $f$-structure of the domain and complex structure of the target. When the codomain is $(1,2)$-symplectic, that is the K\"ahler form $\omega$ satisfies $(d\omega)^{1,2} =0$, the harmonicity of a pseudo horizontally weakly conformal map is characterized by
$$F\Div F =0.$$
Harmonic pseudo horizontally weakly conformal maps are called pseudo harmonic morphisms and one can easily show that they pull back (local) holomorphic functions onto harmonic functions.

A subclass of pseudo horizontally weakly conformal maps, called pseudo horizontally homothetic maps and defined by
$$d\varphi_{x}((\nabla_{v}d\varphi^{t}(JY))_{x}) = J_{\varphi(x)}d\varphi_{x}((\nabla_{v}d\varphi^{t}(Y))_{x}) ,$$
was introduced by Aprodu, Aprodu and Brinzanescu~\cite{AAB} to construct new maps preserving minimal submanifolds. Recently, in a sort of return to the origins, harmonic pseudo horizontally homothetic submersions were shown to be weakly stable~\cite{Aprodu}. While pseudo horizontal weak conformality is a compatibility condition of the complex structure and the metric, pseudo horizontal homothety is a partial K\"ahler condition on the horizontal bundle.

The principal difference between the conditions of horizontally weakly conformal and pseudo horizontally weakly conformal is that the latter does not connect the domain and target metrics together. This extra flexibility, for example, non-constant immersions can be pseudo horizontally weakly conformal, is also the source of new difficulties when trying to put in evidence geometric properties. To avoid these limitations, this article takes the option of concentrating on a fundamental constituent of a map: the eigenvalues of its first fundamental form, which is the pull-back of the target metric.
Computing their derivatives, we are able to give an expression of harmonicity, a characterization of totally geodesic submersions and a criterion of invariance under biconformal change of metric. When looking at the distribution of the eigenvalues, the filiation with holomorphic maps appears clearly and we can obtain a Schwarz lemma, that is geometric conditions on the domain and codomain restricting the growth, or even the existence, of pseudo harmonic morphisms. The approach is essentially based on~\cite{GH} but is also reminiscent of~\cite{Yau}. Finally, working from a five dimensional Riemannian manifold into a four dimensional Hermitian manifold, we observe that the first fundamental form of pseudo harmonic morphisms can only admit two non-zero eigenvalues and derive a formula for the Laplacian of their difference. This enables us to find conditions on the spaces and the map forcing these two eigenvalues to agree and the map to be a harmonic morphism.

\section{Derivatives of eigenvalues}

Let $\varphi: (M^{m}, g) \to (N^{n}, h)$ be a smooth map between Riemannian manifolds. 
The first fundamental form $\varphi^{*}h$ is a symmetric, positive semi-definite covariant $2$-tensor field on $M$ 
defined by the pull-back of the metric on $N$.
Via the musical isomorphisms, it can be seen as the endomorphism 
$d \varphi^{t} \circ d \varphi: TM \to TM$, where $d \varphi^{t}:TN \to TM$ is the adjoint
of $d \varphi$.

Since it is symmetric, the first fundamental form is diagonalizable and we let $\lambda_{1}^{2} \geq \cdots \geq \lambda_{r}^{2} \geq \lambda_{r+1}^{2}=\cdots=\lambda_{m}^{2}=0$
be its (real, non-negative) eigenvalues, where $r = \rank (d \varphi)$. 

A non-constant map is horizontally weakly conformal when all the non-zero eigenvalues of $\varphi^{*}h$ are equal.

A vector $X \in T_{x}M$ will be called an eigenvector of $\varphi^{*}h_{x}$ if
$$\varphi^{*}h_{x}(X, Y)=\lambda^{2}g_{x}(X,Y), \quad  \forall Y \in T_{x}M ,$$
or equivalently
$$d \varphi^{t}_{x} \circ d \varphi_{x}(X)=\lambda^{2}X.$$

At any point of $M$, there exists an orthonormal basis of eigenvectors $\{ E_{i} \}_{i=1,\dots,m}$ and, around any point of a dense open subset of $M$,
we have a local orthonormal frame of eigenvector fields (\cite{PW}).

In particular, for eigenvector fields we have
$$\varphi^{*}h(E_{i}, E_{j})=\delta_{ij}\lambda_{i}^{2},$$ 
so the vectors $ d \varphi(E_{i})$ ($i=1,\dots,m$) are orthogonal of norm $|d \varphi(E_{i})|=\lambda_{i}$ and, 
decomposing $d \varphi(E_{i})=\sum_{a=1}^{n}\varphi_{i}^{a}Y_{a}$ in an orthonormal frame $\{ Y_{a} \}_{a=1,\dots,n}$
in $TN$, we have
$$
\lambda_{i}^{2}=h(d \varphi(E_{i}), d \varphi(E_{i}))=\sum_{a=1}^{n}\varphi_{i}^{a}\varphi_{i}^{b}h_{ab}=\sum_{a=1}^{n}(\varphi_{i}^{a})^{2}.
$$

\begin{lemma}
Let $\varphi: (M^{m}, g) \to (N^{n}, h)$ be a smooth map, $\{\lambda_{i}\}_{i=1,\dots,m}$ its eigenvalues and $\{E_{i} \}_{i=1,\dots,m}$ an 
orthonormal frame of eigenvectors of $\varphi^{*}h$. Then the derivative of $\lambda_{i}^2$ in the direction of $E_{k}$ is
\begin{equation}\label{dlambda}
E_{k}(\lambda_{i}^{2})= 2h(\nabla d \varphi (E_{k},E_{i}), d \varphi (E_{i})), 
\end{equation}
for all $i,k= 1,\dots,m$, or alternatively
\begin{equation}\label{dlambdat}
E_{k}(\lambda_{i}^{2}) = -2h( d \varphi (E_{k}), \nabla d \varphi (E_{i}, E_{i}))+
2(\lambda_{i}^{2}-\lambda_{k}^{2})g(E_{k},\nabla^{M}_{E_{i}} E_{i}), 
\end{equation}
for all $i,k= 1,\dots,m$, with $i\neq k$.
\end{lemma}
\begin{proof}
Let $\lambda_{i}$ be the i-th eigenvalue of $\varphi^{*}h$ and $E_{k}$ an eigenvector corresponding to the eigenvalue $\lambda_{k}$.
\begin{align*}
E_{k}(\lambda_{i}^{2})&=E_{k}(\varphi^{*}h(E_{i}, E_{i}))\\
&=E_{k}\left(h(d \varphi (E_{i}), d \varphi (E_{i}))\right)\\
&=2h(\nabla^{\varphi} _{E_{k}} d \varphi (E_{i}), d \varphi (E_{i}))\\
&=2h(\nabla d \varphi (E_{k},E_{i})+ d \varphi \left(\nabla^{M}_{E_{k}} E_{i}\right), d \varphi (E_{i}))\\
&=2h(\nabla d \varphi (E_{k},E_{i}), d \varphi (E_{i}))+ 2\lambda_{i}^{2}g(\nabla^{M}_{E_{k}} E_{i}, E_{i})\\
&=2h(\nabla d \varphi (E_{k},E_{i}), d \varphi (E_{i})), 
\end{align*}
for all $i,k= 1,\dots,m$.\\
A different expression of the derivatives of $\lambda_{i}$ is given by
\begin{align*}
E_{k}(\lambda_{i}^{2})&=E_{k}(\varphi^{*}h(E_{i}, E_{i}))\\
&=E_{k}\left(h(d \varphi (E_{i}), d \varphi (E_{i}))\right)\\
&=2h(\nabla^{\varphi} _{E_{k}} d \varphi (E_{i}), d \varphi (E_{i}))\\
&=2h(\nabla^{\varphi} _{E_{i}} d \varphi (E_{k})+d \varphi([E_{k},E_{i}]), d \varphi (E_{i}))\\
&=-2h( d \varphi (E_{k}), \nabla^{\varphi} _{E_{i}} d \varphi (E_{i}))+2\lambda_{i}^{2}g([E_{k},E_{i}], E_{i})\\
&=-2h( d \varphi (E_{k}), \nabla d \varphi (E_{i}, E_{i})+d \varphi(\nabla^{M}_{E_{i}} E_{i}))+
2\lambda_{i}^{2}g(E_{k},\nabla^{M}_{E_{i}} E_{i})\\
&=-2h( d \varphi (E_{k}), \nabla d \varphi (E_{i}, E_{i}))+
2(\lambda_{i}^{2}-\lambda_{k}^{2})g(E_{k},\nabla^{M}_{E_{i}} E_{i}), 
\end{align*}
for all $i,k= 1,\dots,m ,\, i\neq k$.
\end{proof}

As mentioned in \cite{Lemaire} for holomorphic maps, eigenvalues of pseudo horizontally weakly conformal maps come in couples.

\begin{proposition}
Let $\varphi: (M^{m}, g) \to (N^{2n}, J, h)$ be a pseudo horizontally weakly conformal map of rank $2r$ into an almost Hermitian manifold and F its associated metric f-structure induced on M (cf.~\cite{Loubeau}).\\ 
Then $\varphi^{*}h$ is F-invariant and $\varphi$ is $(F,J)$-holomorphic.\\
If $X \in \Gamma(TM)$ is an eigenvector of $\varphi^{*}h$ with respect to $g$, then so is $FX$, for the same eigenvalue.\\
Moreover, let $\spec_{g}(\varphi^{*}h)=\{\lambda_{1}^{2}, ..., \lambda_{r}^{2}, 0 \}$, around any point of M there exists an adapted orthonormal frame $\{E_{i}, F E_{i}, E_{\alpha}\}_{i=1,\dots,r}^{\alpha=1,\dots,m-2r}$ of eigenvectors of $\varphi^{*}h$ such that
the eigenspace of $\lambda_{i}^{2}$ is $\Span \{ E_{i}, FE_{i} \}$ and the eigenspace of the zero eigenvalue is $\Span \{ E_{\alpha} \}_{\alpha=1,\dots,m-2r}$.
\end{proposition}

\begin{proof}
First of all, observe that
\begin{align*}
\varphi^{*}h(FX, Y)&=h(d\varphi(FX), d\varphi(Y))\\
&=h(Jd\varphi(X), d\varphi(Y))\\
&=-h(d\varphi(X), Jd\varphi(Y))\\
&= -h(d\varphi(X), d\varphi(FY))\\
&=-\varphi^{*}h(X, FY) .
\end{align*}
Moreover, supposing that 
$$\varphi^{*}h(X, Y)=\lambda^{2}g(X, Y), \quad \forall Y \in TM,$$ 
then
\begin{equation*}
\varphi^{*}h(FX, Y)=-\varphi^{*}h(X, FY)=-\lambda^{2}g(X, FY)=\lambda^{2}g(FX, Y).
\end{equation*}
To construct an adapted frame, we can use the standard diagonalization procedure once we have checked that, if $E_{1}, FE_{1}$ are orthonormal eigenvector fields associated to $\lambda_{1}^{2} \in \spec_{g}(\varphi^{*}h)$, then $d \varphi^{t} \circ d \varphi$ maps $\Span \{ E_{1}, FE_{1} \}^{\perp}$ onto itself.
Since $d \varphi^{t} \circ d \varphi$ is the endomorphism associated to $\varphi^{*}h$, for 
$Y \in \Span \{ E_{1}, FE_{1} \}^{\perp}$, we have
\begin{equation*}
g(d \varphi^{t} \circ d \varphi(Y), E_{1}) = \varphi^{*}h(Y, E_{1})= \lambda_{1}^{2}g(Y, E_{1}) =0,
\end{equation*}
and similarly for $FE_{1}$.
\end{proof}

This enables a particular expression for the derivatives of eigenvalues of the first fundamental form of a pseudo horizontally weakly conformal mapping.

\begin{proposition}
Let $\varphi: (M^{m}, g) \to (N^{2n}, J, h)$ be a pseudo horizontally weakly conformal map of rank $2r$ into an almost Hermitian manifold and $\{E_{i}, F E_{i}, E_{\alpha}\}_{i=1,\dots,r}^{\alpha=1,\dots,m-2r}$ an adapted orthonormal frame of eigenvectors of $\varphi^{*}h$. 
Then, for $E_{k}\in\Gamma((\ker d\varphi)^{\perp}), k \neq i$
\begin{align}\label{elamb}
&(\lambda_{i}^{2}- \lambda_{k}^{2}) g(\nabla_{E_{i}}E_{i}+\nabla_{FE_{i}}FE_{i}, E_{k})=\\
&3(d \varphi^{*}\Omega) (E_{i}, FE_{i}, E_{k})+E_{k}(\lambda_{i}^{2})
-\lambda_{k}^{2} g\left(F[(\nabla_{E_{i}}F)E_{i}+(\nabla_{FE_{i}}F)FE_{i}], E_{k}\right), \notag
\end{align}
where $\Omega$ is the fundamental 2-form on $(N, J, h)$, i.e. $\Omega(X, Y) = h(X, JY)$.
\end{proposition}
\begin{proof}
This relation can be computed directly or combining \eqref{dlambdat} and 
\begin{equation} \label{nabladfi}
\left( \nabla^{\varphi}_{X}J \right) d \varphi (Y) =
d \varphi \left( (\nabla_{X}F)Y \right)+\nabla d \varphi (X, FY)- 
J \nabla d \varphi (X,Y),  
\end{equation}
for all $X,Y \in \Gamma((\ker d\varphi)^{\perp})$.
\end{proof}

\section{Applications to almost submersions}

For almost submersions, the derivatives of the eigenvalues encode information on the horizontal distribution and harmonicity.

\begin{proposition}
Let $\varphi: (M^{m}, g) \to (N^{n}, h)$ be an almost submersion, with $\VV = \ker d\varphi$ its vertical distribution and $\HH=\VV^{\perp}$, its horizontal distribution. 
Let $\{E_{i}, E_{\alpha} \}_{i=1,\dots,n}^{\alpha=1,\dots,m-n}$ be an  
orthonormal frame of eigenvectors of $\varphi^{*}h$ such that $E_{i}\in \HH, \, \forall i= 1,\dots,n$ and $E_{\alpha} \in \VV , \, \forall \alpha=1,\dots,m-n$.
Then, at regular points, the mean curvature $\mu^{\HH}$ of the horizontal distribution is given by
$$\mu^{\HH}=\frac{1}{n}\grad^{\VV}\ln(\lambda_{1} \cdots \lambda_{n}).$$
In particular for a horizontally weakly conformal map, we recover the expression
$$\mu^{\HH}=\grad^{\VV}\ln\lambda .$$
The map $\varphi$ is harmonic if and only if
\begin{equation}\label{harm-gen}
E_{k}[e(\varphi)-\lambda_{k}^{2}]=\sum_{i=1}^{n}(\lambda_{i}^{2}- \lambda_{k}^{2}) g(\nabla_{E_{i}}E_{i}, E_{k})
-(m-n) \lambda_{k}^{2} g(\mu^{\VV}, E_{k}), 
\end{equation}
for all $k=1,\dots,n$, where $e(\varphi)$ is the energy density of $\varphi$.
\end{proposition}

\begin{proof}
If $\varphi: (M^{m}, g) \to (N^{n}, h)$ is an almost submersion and 
$\{E_{i}, E_{\alpha} \}_{i=1,\dots,n}^{\alpha=1,\dots,m-n}$ is an  
orthonormal frame of eigenvectors of $\varphi^{*}h$, then $d \varphi (E_{\alpha})=0$ and $\lambda_{\alpha}=0$ and, 
from \eqref{dlambdat}, we have
\begin{equation}\label{nei}
E_{\alpha}(\lambda_{i}^{2})=2\lambda_{i}^{2}g(\nabla^{M}_{E_{i}}E_{i}, E_{\alpha}),
\end{equation}
hence
\begin{equation*}
E_{\alpha}(\ln\lambda_{i})=g(\nabla^{M}_{E_{i}}E_{i}, E_{\alpha}).
\end{equation*}
Therefore $\left(\nabla^{M}_{E_{i}}E_{i}\right)^{\VV}=\grad^{\VV}\ln\lambda_{i}$, and
\begin{equation*} %\label{muh}
\mu^{\HH}=\frac{1}{n}\grad^{\VV}\ln(\lambda_{1} \cdots \lambda_{n}).
\end{equation*}
We characterize the harmonicity of $\varphi$ with the divergence of the stress-energy tensor
$$
S_{\varphi}=e(\varphi)g - \varphi^{*}h,
$$
since
\begin{equation}\label{divh}
\mathrm{div}S_{\varphi}=d e(\varphi) - \mathrm{div}\varphi^{*}h.
\end{equation}
Then
\begin{align*}
\left(\mathrm{div}\varphi^{*}h\right)(E_{k})&=\sum_{i=1}^{m}(\nabla_{E_{i}}\varphi^{*}h)(E_{i}, E_{k})\\
&= \sum_{i=1}^{m}E_{i}(\varphi^{*}h(E_{i}, E_{k}))-\varphi^{*}h(\nabla_{E_{i}}E_{i}, E_{k})-\varphi^{*}h(E_{i}, \nabla_{E_{i}}E_{k})\\
&=\sum_{i=1}^{m}E_{i}(\lambda_{i}^{2}\delta_{ik})-\lambda_{k}^{2}g(\nabla_{E_{i}}E_{i}, E_{k})-\lambda_{i}^{2}g(E_{i}, \nabla_{E_{i}}E_{k})\\
&=E_{k}(\lambda_{k}^{2})+\sum_{i=1}^{n}(\lambda_{i}^{2}-\lambda_{k}^{2})g(\nabla_{E_{i}}E_{i}, E_{k}).
\end{align*}
Since $\varphi$ is submersive on a dense set, its harmonicity is equivalent to
\begin{equation}\label{divo}
\left(\mathrm{div}S_{\varphi}\right)(E_{k})=0 , \quad \forall k=1,\dots,n, 
\end{equation}
that is
$$E_{k}[e(\varphi)-\lambda_{k}^{2}]=\sum_{i=1}^{m}(\lambda_{i}^{2}- \lambda_{k}^{2}) g(\nabla_{E_{i}}E_{i}, E_{k}), $$
for all $k=1,\dots,n$.
Decomposing the sum in vertical and horizontal parts, yields the result.
\end{proof}

\begin{remark}
\begin{enumerate}
\item A computation similar to~\eqref{nei}, for eigenvectors $X$ and $Y$, corresponding to the same eigenvalue $\lambda$ of $\varphi^{*}h$, yields
$$
(\mathcal{L}_{V}g)(X, Y) =-V(\ln \lambda^{2})g(X, Y).
$$
\item The proposition can also be obtained summing up \eqref{dlambdat} for $i \neq k$,
taking into account $E_{k}(\lambda_{k}^{2})=2h(\nabla d \varphi (E_{k},E_{k}), d \varphi (E_{k}))$
and adding the term
$-2h( d \varphi (E_{k}), \sum_{\alpha} \nabla d \varphi (E_{\alpha}, E_{\alpha}))$, which equals $(m-n) \lambda_{k}^{2} g(\mu^{\VV}, E_{k})$.
\item Some straightforward consequences of \eqref{divh} are
\begin{enumerate}
\item A harmonic map has constant energy density 
if and only if its first fundamental form is conservative (i.e. $\mathrm{div}\varphi^{*}h=0$);
\item A harmonic map has minimal fibres if and only if $\mathrm{div}^{\HH} \varphi^{*}h = d e(\varphi)$.
\end{enumerate}
\end{enumerate}
\end{remark}

It is well-known that for horizontally weakly conformal maps, i.e. $\lambda_{1}^{2}= \cdots =\lambda_{n}^{2}=\lambda^{2}$, Condition~\eqref{harm-gen} becomes
$$
(n-2)\mathrm{grad}^{\HH}(\mathrm{ln} \lambda)=-(m-n)\mu^{\VV}.
$$

The special formula for the derivatives of the eigenvalues of the first fundamental form of a pseudo horizontally weakly conformal almost submersion leads to a re-writing of the condition of harmonicity.

\begin{corollary}
A pseudo horizontally weakly conformal almost submersion $\varphi$ is harmonic if and only if the eigenvalues and eigenvector fields of $\varphi^{*}h$ satisfy
\begin{align*}
&3\sum_{i=1}^{n}(d \varphi^{*}\Omega) (E_{i}, FE_{i}, E_{k})-
\lambda_{k}^{2} g(F\mathrm{div}^{\HH}F+(m-n)\mu^{\VV}, E_{k})=0,\\
&3\sum_{i=1}^{n}(d \varphi^{*}\Omega) (E_{i}, FE_{i}, FE_{k})-
\lambda_{k}^{2} g(F\mathrm{div}^{\HH}F+(m-n)\mu^{\VV}, FE_{k})=0, 
\end{align*}
for all $k=1,\dots,n$, where $\Omega(X,Y) = h(X,JY)$, or equivalently
\begin{align*}
&3\sum_{i=1}^{n}\left[(d \varphi^{*}\Omega) (E_{i}, FE_{i}, E_{k})-\lambda_{k}^{2}d \Phi (E_{i}, FE_{i}, E_{k})\right]
-\lambda_{k}^{2} (m-n)g(\mu^{\VV}, E_{k})=0,\\
&3\sum_{i=1}^{n}\left[(d \varphi^{*}\Omega) (E_{i}, FE_{i}, FE_{k})-\lambda_{k}^{2}d \Phi (E_{i}, FE_{i}, FE_{k})\right]
-\lambda_{k}^{2} (m-n)g(\mu^{\VV}, FE_{k})=0, 
\end{align*}
for all $k=1,\dots,n$, where $\Phi(X,Y) = g(X,FY)$.
\end{corollary}

\begin{proof}
According to \eqref{harm-gen}, the harmonicity of a pseudo horizontally weakly conformal almost submersion is given by
\begin{equation}\label{phm}
E_{k}[e(\varphi)- \lambda_{k}^{2}]=\sum_{i=1}^{n}(\lambda_{i}^{2}- \lambda_{k}^{2}) g(\nabla_{E_{i}}E_{i}+\nabla_{FE_{i}}FE_{i}, E_{k})
-(m-n)\lambda_{k}^{2} g(\mu^{\VV}, E_{k}),
\end{equation}
for $k=1,\dots,n$.\\
As the eigenvalues are double, 
$$e(\varphi)=\lambda_{1}^{2}+\cdots +\lambda_{n}^{2},$$
so 
$$E_{k}[e(\varphi)- \lambda_{k}^{2}]=E_{k}\left[\sum_{i \neq k}\lambda_{i}^{2} \right].$$
Observe that $E_{k}(\lambda_{k}^{2})$ is not involved in the harmonicity condition.\\
Summing \eqref{elamb} for $i\neq k$, we obtain
\begin{align*}
&\sum_{i=1}^{n}(\lambda_{i}^{2}- \lambda_{k}^{2}) g(\nabla_{E_{i}}E_{i}+\nabla_{FE_{i}}FE_{i}, E_{k})=\\
&3\sum_{i=1}^{n}(d \varphi^{*}\Omega) (E_{i}, FE_{i}, E_{k})+E_{k}[e(\varphi)-\lambda_{k}^{2}]
-\lambda_{k}^{2} g(F\mathrm{div}^{\HH}F, E_{k}),
\end{align*}
and, replacing in \eqref{phm}, we get the proposition.
\end{proof}

\begin{remark}
The harmonicity of a pseudo horizontally weakly conformal almost submersion with values in a (1,2)-symplectic manifold is given by
the criterion (\cite{LM})
$$
F\mathrm{div}^{\HH}F+(m-n)\mu^{\VV}=0.
$$
Notice that, in this case, harmonicity is intrinsic to the $f$-manifold $(M, F, g)$, 
i.e. it can be formulated independently of the first fundamental form of the map, as the left-hand term is equal to $F\mathrm{div}F$.
This is no longer true when the target manifold is not (1,2)-symplectic.
\end{remark}

Clearly, by Equation~\eqref{dlambda}, totally geodesic maps have constant eigenvalues.
\begin{proposition}
A submersion is totally geodesic if and only if its first fundamental form is parallel. Consequently, 
its eigenvalues are constant.
\end{proposition}

\begin{proof}
One implication was proved in \cite{BE}.\\
The proposition relies on the following relationship:
a map $\varphi: (M,g)\to (N,h)$ between Riemannian manifolds, satisfies the identity
\begin{equation}\label{mag1}
\left( \nabla_X \varphi^* h \right)(Y, Z)= h(\nabla d \varphi(X, Y), d \varphi(Z)) + h(d \varphi(Y), \nabla d \varphi(X,Z)).
\end{equation}
and its counterpart
\begin{align}\label{mag2}
&h(\nabla d\varphi(X, Y), d \varphi(Z))\\
& =
\frac{1}{2}\left[\left( \nabla_X \varphi^* h \right)(Y, Z)+
\left( \nabla_Y \varphi^* h \right)(Z, X)-
\left( \nabla_Z \varphi^* h \right)(X,Y)\right].\notag
\end{align}
Indeed,
\begin{align*}
&\left( \nabla_X \varphi^* h \right)(Y, Z) \\
&=X \left(\varphi^* h(Y,Z) \right) 
- \varphi^* h(\nabla_{X}Y, Z)
- \varphi^* h(Y, \nabla_{X}Z)\\
&= X \left(h(d \varphi(Y),d \varphi(Z)) \right) 
- h(d \varphi(\nabla_{X}Y), d \varphi(Z))
- h(d \varphi(Y), d \varphi(\nabla_{X}Z))\\
&= h(\nabla^{\varphi}_{X}d \varphi(Y) - d \varphi(\nabla_{X}Y),d \varphi(Z)) 
+ h(d \varphi(Y), \nabla^{\varphi}_{X}d \varphi(Z)-d\varphi(\nabla_{X}Z)),
\end{align*}
that gives us \eqref{mag1}.

Using \eqref{mag1} three times, we get \eqref{mag2}.
\end{proof}

If $\varphi: (M^{m}, g) \to (N^{n}, h)$ a harmonic submersion, we can perturb the metric $g$ biconformally
$$
\ov{g}=\sigma^{-2}g^{\HH}+\rho^{-2}g^{\VV},
$$
where $\sigma$ and $\rho$ are functions on $M$.

\begin{theorem}
Let $\varphi: (M^{m}, g) \longrightarrow (N^{n}, h)$ a smooth almost submersion between Riemannian manifolds. Then any two of the following conditions imply the third
\begin{enumerate}
\item $\varphi$ is harmonic, with respect to the metric $g$;
\item $\varphi$ is harmonic, with respect to the metric $\ov g$;
\item $\grad^{\HH} \rho^{m-n}\sigma^{n-2}=0$.
\end{enumerate}
\end{theorem}

\begin{proof}
Since $\varphi^{*}h(E_{i}, Y)=\lambda_{i}^{2}g(E_{i}, Y)=\lambda_{i}^{2}\sigma^{2}\ov{g}(E_{i}, Y)$, with respect to the new metric $\ov{g}$, the eigenvalues and adapted frames of eigenvectors of $\varphi^{*}h$ are given by
$$\ov{\lambda_{i}}^{2}=\sigma^{2} \lambda_{i}^{2}; \quad  \ov{E}_{i}=\sigma E_{i}; \quad \ov{E}_{\alpha}=\rho E_{\alpha}.
$$

To check Equation~\eqref{harm-gen} with respect to $\ov{g}$, we compute each term separately
\begin{align*}
\ov{g}\left(\ov{\nabla}_{\ov{E}_{i}}\ov{E}_{i}, \ov{E}_{k}\right)&=-\ov{g}(\ov{E}_{i}, [\ov{E}_{i}, \ov{E}_{k}])\\
&=-\sigma^{-2} g(\sigma E_{i},\  [\sigma E_{i}, \sigma E_{k}])\\
&=-\sigma^{-2} g(\sigma E_{i},\  \sigma^{2}[E_{i}, E_{k}]+\sigma E_{i}(\sigma)E_{k}-\sigma E_{k}(\sigma)E_{i})\\
&=- \sigma g(E_{i}, [E_{i}, E_{k}])+ E_{k}(\sigma)\\
&=\sigma g(\nabla_{E_{i}}E_{i}, E_{k})+ E_{k}(\sigma)
\end{align*}
and
\begin{align*}
\ov{g}\left(\ov{\nabla}_{\ov{E}_{\alpha}}\ov{E}_{\alpha}, \ov{E}_{k}\right)&=-\ov{g}(\ov{E}_{\alpha}, [\ov{E}_{\alpha}, \ov{E}_{k}])\\
&=-\rho^{-2} g(\rho E_{\alpha}, [\rho E_{\alpha}, \sigma E_{k}])\\
&=-\rho^{-2} g(\rho E_{\alpha},\  \sigma \rho [E_{\alpha}, E_{k}]+\rho E_{\alpha}(\sigma)E_{k}-\sigma E_{k}(\rho)E_{\alpha})\\
&=- \sigma g(E_{\alpha}, [E_{\alpha}, E_{k}])+ \sigma \rho^{-1}E_{k}(\rho)\\
&=\sigma g(\nabla_{E_{\alpha}}E_{\alpha}, E_{k})+ \sigma E_{k}(\mathrm{ln}\rho).
\end{align*}
Using these formulas in \eqref{harm-gen} with $\ov{g}$, we have
\begin{align*}
&\ov{E}_{k}[\ov{e}(\varphi)-\ov{\lambda}_{k}^{2}]=\sum_{i=1}^{n}(\ov{\lambda}_{i}^{2}- 
\ov{\lambda}_{k}^{2}) \ov{g}(\ov{\nabla}_{\ov{E}_{i}}\ov{E}_{i}, \ov{E}_{k}) 
-(m-n) \ov{\lambda}_{k}^{2} \ov{g}(\ov{\mu}^{\VV}, \ov{E}_{k}); \\
&\sigma E_{k}[\sigma^{2}(e(\varphi)-\lambda_{k}^{2})]=\sum_{i=1}^{n}\sigma^{2}(\lambda_{i}^{2}- \lambda_{k}^{2})
\left[\sigma g(\nabla_{E_{i}}E_{i}, E_{k})+ E_{k}(\sigma)\right]\\
&-(m-n) \sigma^{2} \lambda_{k}^{2} \left[\sigma g(\mu^{\VV}, E_{k})+ \sigma E_{k}(\ln\rho)\right]; \\
&E_{k}(\ln\sigma^{2})(e(\varphi)-\lambda_{k}^{2})+E_{k}[e(\varphi)-\lambda_{k}^{2}]\\
&=\sum_{i=1}^{n}(\lambda_{i}^{2}- \lambda_{k}^{2}) \left[g(\nabla_{E_{i}}E_{i}, E_{k})+ E_{k}(\ln \sigma)\right]-(m-n) \lambda_{k}^{2} \left[g(\mu^{\VV}, E_{k})+ E_{k}(\ln\rho)\right].
\end{align*}
As $\varphi$ is harmonic with respect to $g$, the condition simplifies to
\begin{align*}
E_{k}(\ln\sigma^{2})(e(\varphi)-\lambda_{k}^{2})=
&\sum_{i=1}^{n}(\lambda_{i}^{2}- \lambda_{k}^{2})E_{k}(\ln \sigma)-
(m-n) \lambda_{k}^{2} E_{k}(\ln\rho),
\end{align*}
but
\begin{align*}
&2E_{k}(\ln\sigma) e(\varphi)-2\lambda_{k}^{2}E_{k}(\ln\sigma)\\
&-2 e(\varphi) E_{k}(\ln\sigma)+ n\lambda_{k}^{2}E_{k}(\ln \sigma) +(m-n) \lambda_{k}^{2} E_{k}(\ln\rho) \\
&= \lambda_{k}^{2}[( n-2)E_{k}(\ln \sigma)+ (m-n) E_{k}(\ln\rho)] \\
&= \lambda_{k}^{2} E_{k}(\ln \sigma^{n-2}\rho^{m-n}).
\end{align*}
\end{proof}

\begin{remark}
\begin{enumerate}
\item Compare with \cite[Corollary 4.6.10]{BW} and \cite{Slobodeanu}.
\item The implications $(1) \wedge (3) \Rightarrow (2)$ and $(2) \wedge (3) \Rightarrow (1)$ have already been proved, in a different way, in \cite{Pantilie}.
\end{enumerate}
\end{remark}

\section{A Schwarz lemma for pseudo harmonic morphisms}

Let $\varphi: (M^{m}, g) \to (N^{n}, h)$ be a smooth map of rank $r \leq k=\mathrm{min}(m,n) $. Denote by 
$\wedge ^{p} d\varphi$ the induced map on $p$-vectors, its norm can be expressed in terms of the eigenvalues of the first fundamental form
$\lambda_{1}^{2}\geq\lambda_{2}^{2}\geq\cdots\geq \lambda_{r}^{2}$, by
$$\norm{\wedge ^{p} d\varphi}^2 =\sum_{i_{1} < \cdots <i_{p}}\lambda_{i_{1}}^{2}\cdots \lambda_{i_{p}}^{2}.$$
Moreover, for $p=2$, we have the inequality
\begin{equation*}%\label{funding}
\norm{\wedge ^{2} d\varphi} \leq \sqrt{\frac{k-1}{2k}} \norm{d \varphi}^{2},
\end{equation*}
with equality if and only if $r=k$ and $\lambda_{1}^{2}=\lambda_{2}^{2}= \cdots = \lambda_{r}^{2}$.

\begin{definition}\cite{Goldberg}
A map $\varphi$ is said to be of bounded dilatation of order $K$ if the ratio of the first two eigenvalues of $\varphi^{*}h$ is bounded by $K^2$, i.e.
$$\lambda_{1}^{2}/\lambda_{2}^{2} \leq K^{2}.$$
\end{definition}
For a map of bounded dilatation of order $K$, one can easily establish the reverse inequality (\cite{GH})
\begin{equation*}%\label{kbd}
\norm{d\varphi}^{2} \leq k K \norm{\wedge ^{2} d\varphi}.
\end{equation*}

If $M$ is a complete manifold with non-positive sectional curvature and Ricci curvature bounded from below by a negative constant $-A$ and 
the sectional curvature of the manifold $N$ is bounded above by a negative constant $-B$, Goldberg and Har'el showed in ~\cite{GH} that any harmonic map $\varphi : M \to N$ of bounded dilatation of order $K$ satisfies
$$ \norm{\wedge^r d\varphi}^{2/r} \leq (k/2)\binom{k}{r}^{1/r} \frac{A}{B} K^2,$$
for any $r$, $1\leq r \leq k$.
In particular, for $r=1$
\begin{equation*}%\label{ebd}
e(\varphi) \leq k^{2} K^{2} \frac{A}{2B}.
\end{equation*}
If $\mathrm{Ric}^{M} \geq 0$, then $\varphi$ is constant.

The binary nature of the eigenvalues of pseudo horizontally weakly conformal maps, enables the extension of some results previously known for holomorphic maps.

\begin{proposition}
\begin{itemize}
\item\cite{GH} Any pseudo horizontally weakly conformal map in an almost Hermitian manifold is of bounded dilatation of order 1.
\item\cite{Lemaire} If M and N are compact manifolds and N is (1,2)-symplectic of negative sectional curvature, then there exists 
only a finite number of non-constant pseudo harmonic morphisms from M to N.
\item\cite{EL} If M is a compact manifold and N is a (1,2)-symplectic manifold of nonpositive sectional curvature, 
then two homotopic pseudo harmonic morphisms from M to N which agree at a point are identical.
\end{itemize}
\end{proposition}

\begin{lemma}\label{lem1}
A pseudo horizontally weakly conformal map $\varphi : (M,g) \to (N^{2n},J,h)$ satisfies
\begin{equation*}
\frac{\norm{d\varphi}^{2}}{\norm{\wedge ^{2} d\varphi}} < 2.
\end{equation*}
Moreover, if $\ell_{n-1}=\tfrac{\lambda_{1}}{\lambda_{n}}$, then
\begin{equation*}
\frac{\norm{d\varphi}^{2}}{\norm{\wedge ^{2} d\varphi}} \leq 2 \sqrt{1- \frac{n-1}{2n-1}\ell_{n-1}^{-4}},
\end{equation*}
with equality if and only if the map is horizontally weakly conformal, that is if $\lambda_{1}^{2}= \cdots = \lambda_{n}^{2}$, so $\ell_{n-1}=1$.
\end{lemma}

\begin{proof}
Recall that
$$
\frac{\norm{d \varphi}^{2}}{\norm{\wedge ^{2} d \varphi}}=
\frac{\sum_{i=1}^{n} \lambda_{i}^{2}}{\sqrt{\sum_{1\leq i < j\leq n} \lambda_{i}^{2}\lambda_{j}^{2}}},
$$
since the eigenvalues are double
$$
\frac{\norm{d \varphi}^{2}}{\norm{\wedge ^{2} d \varphi}}=\frac{2\sum_{i=1}^{n} \lambda_{i}^{2}}{\sqrt{\sum_{i=1}^{n} \lambda_{i}^{4}
+4\sum_{1 \leq i < j \leq n} \lambda_{i}^{2}\lambda_{j}^{2}}}.
$$
But, by Newton's inequalities 
$$
\sum_{i=1}^{n} \lambda_{i}^{4}+4\sum_{1 \leq i < j \leq n} \lambda_{i}^{2}\lambda_{j}^{2} = 
\left( \sum_{i=1}^{n} \lambda_{i}^{2} \right)^{2}+ 2\sum_{1 \leq i < j \leq n} \lambda_{i}^{2}\lambda_{j}^{2}
\leq \frac{2n-1}{n} \left( \sum_{i=1}^{n} \lambda_{i}^{2} \right)^{2},
$$ 
so
\begin{align*}
\frac{\norm{d \varphi}^{2}}{\norm{\wedge ^{2} d \varphi}}=
&2\sqrt{1- 
\frac{2\sum_{1 \leq i < j \leq n} \lambda_{i}^{2}\lambda_{j}^{2}}{\left( \sum_{i=1}^{n} \lambda_{i}^{2} \right)^{2}+ 2\sum_{1 \leq i < j \leq n} \lambda_{i}^{2}\lambda_{j}^{2}}}\\
\leq & 2\sqrt{1- 
\frac{2n}{2n-1} \frac{\sum_{1 \leq i < j \leq n} \lambda_{i}^{2}\lambda_{j}^{2}}{\left( \sum_{i=1}^{n} \lambda_{i}^{2} \right)^{2}}}.
\end{align*}
Since
$$
\frac{\left( \sum_{i=1}^{n} \lambda_{i}^{2} \right)^{2}}{\sum_{1 \leq i < j \leq n} \lambda_{i}^{2}\lambda_{j}^{2}} \leq \frac{2n}{n-1} \ell_{n-1}^{4},
$$
we obtain the lemma.
\end{proof}

With this lemma, we can improve the bound on the energy density of a pseudo harmonic morphism.

\begin{theorem}
Let $(M^{m}, g)$ be a complete Riemannian manifold with non-positive sectional curvature and Ricci curvature bounded from below by a negative constant $-A$ and $(N^{2n},J, h)$ a $(1,2)$-symplectic almost Hermitian manifold of sectional curvature bounded above by a negative constant $-B$
If $\varphi : M \to N$ is a pseudo harmonic morphism then 
\begin{equation*} 
 e(\varphi) \leq \frac{A}{B}.
\end{equation*}
If $B \geq 2A$, then $\varphi$ is distance-decreasing.
In particular, if $\mathrm{Ric}^{M} \geq 0$, then $\varphi$ is constant.
If the $(n-1)^{th}$ dilatation is bounded, that is $\ell_{n-1}^{2}=\frac{\lambda_{1}^{2}}{\lambda_{n}^{2}} \leq L^{2}$, we have
\begin{equation*} 
e(\varphi) \leq \frac{A}{B} \left( 1- \frac{n-1}{2n-1}L^{-4} \right).
\end{equation*}
\end{theorem}

\begin{proof}
We follow the argumentation of \cite{GH}. Let $\varphi : M \to N$ be a pseudo harmonic morphism, that is a pseudo horizontally weakly conformal harmonic map. In \cite{Goldberg}, Goldberg uses the Weitzenb\"ock formula for harmonic maps and an exhaustion of $M$ by conformal submanifolds $(M_{\rho}, e^{v_{\rho}}g)$, to show that there exists a maximum point $x$ of the function $e^{-2v_{\rho}}\norm{d\varphi}^2$ and a function $\epsilon(\rho)$ with $\lim_{\rho \to +\infty}\epsilon(\rho) =0$ such that
$$- e^{-2v_{\rho}} R^{N}(d\varphi(e_{i}),d\varphi(e_{j}),d\varphi(e_{i}),d\varphi(e_{j})) \leq (A+ \epsilon(\rho)) \norm{d\varphi}^2 ,$$
at the maximum $x$.
The condition on the sectional curvature of $(N,h)$ implies that
$$R^{N}(d\varphi(e_{i}),d\varphi(e_{j}),d\varphi(e_{i}),d\varphi(e_{j})) \leq -2B\norm{\wedge^2 d\varphi}^2.$$
Therefore, at this maximum point
$$2e^{-2v_{\rho}}B\norm{\wedge^2 d\varphi}^2 \leq (A+ \epsilon(\rho)) \norm{d\varphi}^2.$$
Combining with Lemma~\ref{lem1}, we obtain
$$\frac{1}{2}e^{-2v_{\rho}}B\norm{ d\varphi}^4 \leq (A+ \epsilon(\rho)) \norm{d\varphi}^2,$$
that is
$$\frac{1}{2}e^{-2v_{\rho}}B\norm{ d\varphi}^2 \leq (A+ \epsilon(\rho)),$$
at the maximum point $x$, hence everywhere. Let $\rho\to +\infty$, then $v_{\rho}$ goes to zero and we obtain
$$\norm{ d\varphi}^2 \leq 2\frac{A}{B}.$$
\end{proof}

Since the unit disk has a K\"ahler metric with constant negative holomorphic sectional curvature, we have the following

\begin{corollary} 
Let $(M^{m}, g)$ be a complete Riemannian manifold with non-positive sectional curvature and Ricci curvature bounded from below by a negative constant $-A$. If $M$ is equipped with an $f$-structure then any bounded $f$-holomorphic function on M is constant.
\end{corollary}

\begin{example}[Pseudo harmonic morphisms between space forms]

For a real constant $\kappa$, consider the complex $n$-space $\left( N^{n}(\kappa), h^{(\kappa)} \right)$ which, for 
$\kappa=0, >0$ or $<0$, is $\CC^{n}, \CC P^{n}$ or $\BB^{n}$ (the unit ball in $\CC^{n}$), equipped with the K\"ahler metric
$$
h^{(\kappa)}_{i \ov{\jmath}}=\frac{1}{\zeta}\delta_{ij}-\frac{\kappa}{\zeta^{2}}\ov{z}_{i}z_{j},
$$
where $\zeta = 1+\kappa \abs{z}^{2}$. For $\kappa=0, 1$ and $-1$, we obtain the Euclidean, Fubini-Study and Bergman metrics.

We know that $K(X, JX)=2\kappa$, i.e. the holomorphic sectional curvature is constant and equal to $2\kappa$, and the sectional curvature $K^{N(\kappa)}$ is pinched:
$2\kappa \leq K^{N(\kappa)} \leq \frac{\kappa}{2}$ if $\kappa \leq 0$ (with reversed inequalities for $\kappa \geq 0$).

\medskip
Now consider the analogous Sasakian space forms, the only complete, simply connected Sasakian manifolds with constant 
$\phi$-sectional curvature $c$:

\medskip
$\bullet$ $(\RR^{2n+1}, g=\eta\otimes\eta + \frac{1}{4}\sum_{i=1}^{n}(d x^{i})^{2}+(d y^{i})^{2}, c=-3)$, 
where $\eta=\frac{1}{2}(d z-\sum_{i=1}^{n} y^{i} d x^{i})$;

\medskip
$\bullet$ $(\SS^{2n+1}$, induced metric from $(\CC^{n+1}, h^{(0)}), c=1)$ 

\medskip
$\bullet$ $(\BB^{n} \times \RR, g=\eta\otimes\eta+\pi^{*}h^{(\kappa)}, c=2\kappa-3)$, where $\eta=d t +\pi^{*}\omega$, $d \omega$ is the  K\"ahler 2-form of $\BB^{n}$ and $\pi$ the projection of $\BB^{n} \times \RR$ onto the first factor.

The Ricci tensor of these space forms is \cite{Blair}
$$
\mathrm{Ric}(X,Y)=\frac{n(c+3)+c-1}{2}g(X,Y)-\frac{(n+1)(c-1)}{2}\eta(X)\eta(Y).
$$
In all three cases, $\mathrm{Ric}(\xi, \xi)=2n$. Recall that if we make a $\mathcal{D}$-homothetic transformation of 
the contact structure, for example
$$
\ov{g}=ag+a(a-1)\eta\otimes\eta, \quad a>0,
$$
then the above spaces remain Sasakian space forms, with $\phi$-sectional curvature $\ov{c}=\frac{c+3}{a}-3$.

So in these cases, we have

\medskip
$\bullet$ ($\RR^{2n+1}$):\ $\mathrm{Ric}(E_{i},E_{i})=\mathrm{Ric}(\phi E_{i},\phi E_{i})=-2$.

\medskip
$\bullet$ ($\SS^{2n+1}$):\ $\mathrm{Ric}(E_{i},E_{i})=\mathrm{Ric}(\phi E_{i},\phi E_{i})=\frac{2(n+1-a)}{a}$.

\medskip
$\bullet$ ($\BB^{n} \times \RR$):\ $\mathrm{Ric}(E_{i},E_{i})=\mathrm{Ric}(\phi E_{i},\phi E_{i})=\frac{\kappa(n+1)-2a}{a}$.
\end{example}

From these results, we can conclude the following.

\begin{proposition}
\begin{itemize}
\item Any $(\phi,J)$-holomorphic map $\varphi: \RR^{2n+1} \to \BB^{n}$ 
satisfies $\norm{d \varphi}^{2} \leq \frac{8}{\abs{\kappa}}$. If we take $\kappa \leq -8$, 
then $\varphi$ will be distance decreasing. 
\item There exists only a finite number of $(\phi,J)$-holomorphic maps $\varphi: \SS^{2n+1} \to \BB^{n}$. 
\item Any $(\phi,J)$-holomorphic map $\varphi: \BB^{n} \times \RR \to \BB^{n}$
satisfies $\norm{d \varphi}^{2} \leq \frac{4(-\kappa (n+1)+2a)}{a \abs{\kappa}}$. 
For $\kappa \leq -8$, $\varphi$ will be distance decreasing for any $a>4(n+1)$. 
\end{itemize}
All these maps are pseudo harmonic morphisms.
\end{proposition}
Note, that a holomorphic map from a K\"ahlerian space form to $\BB^{n}$ can be non-constant only if $\kappa < 0$ 
on the domain space, as 
$\mathrm{Ric}(E_{i},E_{i})=\mathrm{Ric}(J E_{i},J E_{i})=(n+1)\kappa$ for any $\left( N^{n}(\kappa), h^{(\kappa)} \right)$.
In this case, a holomorphic map $\varphi: \BB^{m} \to \BB^{n}$ (automatically a pseudo harmonic morphism) 
will satisfy $\norm{d \varphi}^{2} \leq 4(m+1)$.

\section{A Bochner technique for Pseudo Harmonic Morphisms}

Let $\varphi: (M^{5},g) \to (N^{4}, J, h)$ be a pseudo horizontally weakly conformal map from a five-dimensional Riemannian manifold into a 4-dimensional almost Hermitian manifold.

The manifold $M$ inherits an $f$-structure compatible with the metric and, when $M$ is oriented, an almost contact structure (cf.~\cite{Slobodeanu2}).

On an open neighbourhood of a regular point of $\varphi$, consider an adapted frame of unit eigenvectors
$\{ E_{1},FE_{1},E_{2},FE_{2}, V \}$ corresponding to the three eigenvalues $\lambda^{2}_{1} \geq \lambda^{2}_{2} \geq 0$ of $\varphi^{*}h$.
Denote by $\Lambda$ the difference $\lambda^{2}_{1}-\lambda^{2}_{2}$, we shall compute $\Delta \Lambda$ and apply a maximum principle to find conditions forcing the equality of the two eigenvalues, a Bochner technique introduced in~\cite{Baird}.

To unify notations, we shall use the convention $E_{\ov \imath}=FE_{i}$ and $E_{0}=V$. Lower case letters will run over the indices $0, 1, 2$, while upper case letters
will run over the indices $0, 1, 2, \ov 1, \ov 2$. We will also use the notations $\abs {\ov \imath} = i$,  
$\Gamma_{IJ}^{K}=g(\nabla_{E_{I}}E_{J}, E_{K})$ and $R_{IJKL}=R(E_{I}, E_{J}, E_{K}, E_{L})=g( R(E_{I}, E_{J})E_{L}, E_{K})$.

\begin{theorem}
Let $\varphi: (M^{5},g) \to (N^{4}, J, h)$ be an almost submersive pseudo harmonic morphism from a five-dimensional Riemannian manifold into a 4-dimensional almost Hermitian manifold such that $\grad |d\varphi|^2 \in \VV$. Let $p\in M$ be a regular point of $\varphi$, then, in a neighbourhood of $p$, the Laplacian of the function $\Lambda=\lambda^{2}_{1}-\lambda^{2}_{2}$ has the expression
\begin{align*} 
\Delta \Lambda=& 2\Lambda \left( 2\sum_{|I| \neq |K| \neq 0}(\Gamma_{II}^{K})^{2}+
2\sum_{i\neq |K|, K \neq 0}\Gamma_{i i}^{K}\Gamma_{\ov \imath \ov \imath}^{K}+
2\sum_{i, K \neq 0}\Gamma_{i \ov \imath}^{K}\Gamma_{\ov \imath i}^{K} \right) \\
&+2\Lambda \left(R_{1212}+R_{1 \ov 2 1 \ov 2}+R_{\ov 1 2 \ov 1 2}+R_{\ov 1 \ov 2 \ov 1 \ov 2} \right)\\
&+\lambda_{1}^{2} (R_{0101}+R_{0 \ov 1 0 \ov 1})-\lambda_{2}^{2}(R_{0202}+R_{0 \ov 2 0 \ov 2})\\
&+\left( \frac{V(\lambda_{1}^{2})^{2}}{2\lambda_{1}^{2}}-\frac{V(\lambda_{2}^{2})^{2}}{2\lambda_{2}^{2}}
+\Lambda  \frac{V(\lambda_{1}^{2})}{\lambda_{1}^{2}} \frac{V(\lambda_{2}^{2})}{\lambda_{2}^{2}}\right) \\
&-2\lambda_{1}^{2} (\Gamma_{1 \ov 1}^{0})^{2}+
2\lambda_{2}^{2} (\Gamma_{2 \ov 2}^{0})^{2}\\
&+\Lambda \left\{\sum_{\abs I = 1, \abs J = 2}
5(\Gamma_{0I}^{J})^{2}
-2 \left( \frac{1}{\lambda_{1}^{2}}+\frac{1}{\lambda_{2}^{2}} \right)\Gamma_{0I}^{J} \tilde \Gamma_{0I}^{J}
+ \frac{1}{\lambda_{1}^{2}  \lambda_{2}^{2}} (\tilde \Gamma_{0I}^{J})^{2}\right\}\\
&+ 2\lambda_1^2 a_1^2 -2 \lambda_2^2 a_2^2 + 2\lambda_1^2 (a_2 \Gamma_{11}^2 + b_2 \Gamma^{\ov 2}_{11}) 
-  2\lambda_2^2 (a_1 \Gamma_{22}^1 + b_1 \Gamma_{22}^{\ov 1})\\
& + 2 \lambda_1^2 a_1 (\Gamma_{22}^1 + \Gamma_{\ov{22}}^1 ) -  2 \lambda_2^2 a_2 (\Gamma_{11}^2 + \Gamma_{\ov{11}}^2 )\\
&  + 2 \lambda_1^2 b_1 (\Gamma_{22}^{\ov 1} + \Gamma_{\ov{22}}^{\ov 1} ) -  2 \lambda_2^2 b_2 (\Gamma_{11}^{\ov 2} + \Gamma_{\ov{11}}^{\ov 2} )\\
& -2 \lambda_1^2 a_1 \Gamma_{\ov{11}}^1 + 2 \lambda_2^2 a_2 \Gamma_{\ov{22}}^2  + 2\lambda_1^2 FE_{1}(b_1 )  - 2\lambda_2^2 FE_{2}(b_2) ,
\end{align*}
where $\nabla_{V}V= a_{1}E_{1}+a_{2}E_{2}+b_{1}FE_{1}+b_{2}FE_{2}$.
\end{theorem}

\begin{proof}
Let $\varphi: (M^{5},g) \to (N^{4}, J, h)$ a pseudo horizontally weakly conformal map and $p\in M$ a point where the rank is maximum. In a neighbourhood of this point, the first fundamental form $\varphi^{*}h$ admits three eigenvalues $\lambda^{2}_{1}$, $\lambda^{2}_{2}$ and $0$.
Denote by $\Lambda$ the difference $\lambda^{2}_{1}-\lambda^{2}_{2}$ and by $\mu^{\VV}=\nabla_{V}V= a_{1}E_{1}+a_{2}E_{2}+b_{1}FE_{1}+b_{2}FE_{2}$ the mean curvature of the fibres.

By definition
\begin{align*}
\Delta \Lambda&= \Delta^{\VV} \Lambda+\Delta^{\HH} \Lambda\\
&= V(V(\Lambda))-d \Lambda(\nabla_{V}V) +
\sum_{I \neq 0}E_{I}(E_{I}(\Lambda))-d \Lambda(\nabla_{E_{I}}E_{I}) \\
&= V(V(\Lambda))-d \Lambda(\nabla_{V}V) \\
&+ \sum_{i=1,2}E_{i}(E_{i}(\Lambda))-d \Lambda(\nabla_{E_{i}}E_{i})+
FE_{i}(FE_{i}(\Lambda))-d \Lambda(\nabla_{FE_{i}}FE_{i}).
\end{align*}
Since the energy density is constant in horizontal directions, i.e. $\grad{ (e(\varphi))} \in \VV$, then $E_{I}(\lambda_{2}^{2})=-E_{I}(\lambda_{1}^{2})$ and $E_{I}(\Lambda)=2E_{I}(\lambda_{1}^{2})=-2E_{I}(\lambda_{2}^{2}),
\forall I \neq 0$.

To compute $\Delta^{\VV} \Lambda$, we deduce from \eqref{nei} that $V(\lambda^{2}_{i})=2\lambda^{2}_{i}g(\nabla_{E_{i}}E_{i}, V)$ and 
equivalently $V(\lambda^{2}_{i})=2\lambda^{2}_{i}g(\nabla_{FE_{i}}FE_{i}, V)$, so, for $i=1$ or $2$,
\begin{align*}
&V(V(\lambda^{2}_{i}))\\
&=\frac{V(\lambda^{2}_{i})^{2}}{\lambda^{2}_{i}}+2\lambda^{2}_{i}g(\nabla_{V}\nabla_{E_{i}}E_{i}, V)+
2\lambda^{2}_{i}g(\nabla_{E_{i}}E_{i}, \sum_{k=1,2}a_{k}E_{k}+b_{k}FE_{k}),
\end{align*}
Moreover
\begin{align*}
d \Lambda(\nabla_{V}V)&=a_{1}E_{1}(\Lambda)+a_{2}E_{2}(\Lambda)+b_{1}FE_{1}(\Lambda)+b_{2}FE_{2}(\Lambda)\\
&=-2a_{1}E_{1}(\lambda_{2}^{2})+2a_{2}E_{2}(\lambda_{1}^{2})-2b_{1}FE_{1}(\lambda_{2}^{2})+2b_{2}FE_{2}(\lambda_{1}^{2}). 
\end{align*}
From the formula
\begin{align*}
&g(\nabla_{E_{1}}\nabla_{V}V, E_{1})+ g(\nabla_{V}\nabla_{E_{1}}E_{1}, V)=-g(\nabla_{E_{1}}E_{1}, \nabla_{V} V)\\
&-g(\nabla_{E_{1}}V, \nabla_{V} E_{1})
+g(\nabla_{[E_{1}, V]}V, E_{1})+R(E_{1}, V, E_{1}, V),
\end{align*}
we deduce 
\begin{align*}
&g(\nabla_{V}\nabla_{E_{1}}E_{1}, V)\\
&=R(V, E_{1}, V, E_{1})-E_{1}(a_{1})-g(\nabla_{V}E_{1}, \nabla_{E_{1}} V)+
g(\nabla_{[V, E_{1}]}E_{1}, V)\\
&=R_{0101}-E_{1}(a_{1})-\sum_{I= 2,\ov 1, \ov 2} \Gamma_{01}^{I}\Gamma_{10}^{I}+\sum_{J= 0,2,\ov 1, \ov 2}\Gamma_{01}^{J}\Gamma_{J1}^{0}
-\sum_{K=1,2,\ov 1, \ov 2}\Gamma_{10}^{K}\Gamma_{K1}^{0}\\
&=R_{0101}-E_{1}(a_{1})+ a_{1}^{2}+ (\Gamma_{11}^{0})^{2}+ 
\sum_{I= 2,\ov 1, \ov 2} \Gamma_{01}^{I}(\Gamma_{1I}^{0}+\Gamma_{I1}^{0})+\Gamma_{1I}^{0} \Gamma_{I1}^{0}.\\
\end{align*}
Taking into account $\Gamma_{ii}^{0}=\frac{V(\lambda_{i}^{2})}{2\lambda_{i}^{2}}$, we obtain
\begin{align*}
&\Delta^{\VV} \Lambda\\
&=\frac{3}{2} \frac{V(\lambda_{1}^{2})^{2}}{\lambda_{1}^{2}}+
2\lambda_{1}^{2}\left( R_{0101}-E_{1}(a_{1})+ a_{1}^{2}+
\sum_{I\neq1} \Gamma_{01}^{I}(\Gamma_{1I}^{0}+\Gamma_{I1}^{0})+\Gamma_{1I}^{0} \Gamma_{I1}^{0} \right)\\
&-\frac{3}{2} \frac{V(\lambda_{2}^{2})^{2}}{\lambda_{2}^{2}}-
2\lambda_{2}^{2} \left( R_{0202}-E_{2}(a_{2})+ a_{2}^{2}+
\sum_{I\neq2} \Gamma_{02}^{I}(\Gamma_{2I}^{0}+\Gamma_{I2}^{0})+\Gamma_{2I}^{0} \Gamma_{I2}^{0}\right)\\
&+2a_{1}E_{1}(\lambda_{2}^{2})-2a_{2}E_{2}(\lambda_{1}^{2})+2b_{1}FE_{1}(\lambda_{2}^{2})-2b_{2}FE_{2}(\lambda_{1}^{2})\\
&+2\lambda_{1}^{2}(a_{i}\Gamma_{11}^{i}+b_{j}\Gamma_{11}^{\ov \jmath})-
2\lambda_{2}^{2}(a_{i}\Gamma_{22}^{i}+b_{j}\Gamma_{22}^{\ov \jmath}),
\end{align*}

To compute $\Delta^{\HH} \Lambda$, we start with the harmonicity condition \eqref{phm}
$$
E_{1}(\lambda_{2}^{2})=-\Lambda g(\nabla_{E_{2}}E_{2}+\nabla_{FE_{2}}FE_{2}, E_{1})
-\lambda_{1}^{2} g(\mu^{\VV}, E_{1}) ,
$$
hence 
\begin{align*}
&E_{1}(E_{1}(\lambda_{2}^{2}))\\
&= -E_{1}(\Lambda) g(\nabla_{E_{2}}E_{2}+\nabla_{FE_{2}}FE_{2}, E_{1})
-\Lambda g(\nabla_{E_{1}}\nabla_{E_{2}}E_{2}+\nabla_{E_{1}}\nabla_{FE_{2}}FE_{2}, E_{1})\\
&-\Lambda g(\nabla_{E_{2}}E_{2}+\nabla_{FE_{2}}FE_{2}, \nabla_{E_{1}}E_{1})-
a_{1} E_{1}(\lambda_{1}^{2})-\lambda_{1}^{2}E_{1}(a_{1}),
\end{align*}
and similarly for $E_{2}$
\begin{align*}
&E_{2}(\lambda_{1}^{2})=\Lambda g(\nabla_{E_{1}}E_{1}+\nabla_{FE_{1}}FE_{1}, E_{2})
-\lambda_{2}^{2} g(\mu^{\VV}, E_{2}) ;\\ 
&E_{2}(E_{2}(\lambda_{1}^{2}))= E_{2}(\Lambda) g(\nabla_{E_{1}}E_{1}+\nabla_{FE_{1}}FE_{1}, E_{2})\\
&+ \Lambda g(\nabla_{E_{2}}\nabla_{E_{1}}E_{1}+\nabla_{E_{2}}\nabla_{FE_{1}}FE_{1}, E_{2})\\
&+\Lambda g(\nabla_{E_{1}}E_{1}+\nabla_{FE_{1}}FE_{1}, \nabla_{E_{2}}E_{2})-
a_{2} E_{2}(\lambda_{2}^{2})-\lambda_{2}^{2}E_{2}(a_{2}).
\end{align*}
Analogous relations exist for $FE_{1}$ and $FE_{2}$.\\
Finally
\begin{align*}
&d \Lambda(\nabla_{E_{1}}E_{1})\\
=&\ d \Lambda[g(\nabla_{E_{1}}E_{1}, FE_{1})FE_{1}+
g(\nabla_{E_{1}}E_{1}, E_{2})E_{2}+g(\nabla_{E_{1}}E_{1}, FE_{2})FE_{2}\\
&+g(\nabla_{E_{1}}E_{1}, V)V] \\
=&\ \Gamma_{11}^{\ov 1}FE_{1}(\Lambda)+
\Gamma_{11}^{2}E_{2}(\Lambda)+\Gamma_{11}^{\ov 2}FE_{2}(\Lambda)+\Gamma_{11}^{0}V(\Lambda) ,
\end{align*}
with corresponding formulas for $d \Lambda(\nabla_{E_{I}}E_{I}), I = \ov 1, 2, \ov 2$.

Since
\begin{align*} %\label{curvh}
&g(\nabla_{E_{1}}\nabla_{E_{2}}E_{2}, E_{1})+ g(\nabla_{E_{2}}\nabla_{E_{1}}E_{1}, E_{2})=
-g(\nabla_{E_{1}}E_{1}, \nabla_{E_{2}} E_{2})\\
&-g(\nabla_{E_{1}}E_{2}, \nabla_{E_{2}} E_{1})
+g(\nabla_{[E_{1}, E_{2}]}E_{2}, E_{1})+R(E_{1}, E_{2}, E_{1}, E_{2}),
\end{align*}
we conclude that
\begin{align*}
\Delta^{\HH} \Lambda&=- \ 2E_{1}(\lambda_{2}^{2})(\Gamma_{22}^{1}+\Gamma_{\ov 2 \ov 2}^{1})+
2\Lambda  E_{1}(\Gamma_{22}^{1}+\Gamma_{\ov 2 \ov 2}^{1})-2a_{1}E_{1}(\lambda_{2}^{2})+2\lambda_{1}^{2}E_{1}(a_{1})\\
&+\ 2E_{2}(\lambda_{1}^{2})(\Gamma_{11}^{2}+\Gamma_{\ov 1 \ov 1}^{2})+
2\Lambda  E_{2}(\Gamma_{11}^{2}+\Gamma_{\ov 1 \ov 1}^{2})+2a_{2}E_{2}(\lambda_{1}^{2})-2\lambda_{2}^{2}E_{2}(a_{2})\\
&-\ 2FE_{1}(\lambda_{2}^{2})(\Gamma_{22}^{\ov  1}+\Gamma_{\ov 2 \ov 2}^{\ov 1})+
2\Lambda  FE_{1}(\Gamma_{22}^{\ov  1}+\Gamma_{\ov 2 \ov 2}^{\ov 1})-2b_{1}FE_{1}(\lambda_{2}^{2})+2\lambda_{1}^{2}FE_{1}(b_{1})\\
&+\ 2FE_{2}(\lambda_{1}^{2})(\Gamma_{11}^{\ov 2}+\Gamma_{\ov 1 \ov 1}^{\ov  2})+
2\Lambda  FE_{2}(\Gamma_{11}^{\ov 2}+\Gamma_{\ov 1 \ov 1}^{\ov 2})+2b_{2}FE_{2}(\lambda_{1}^{2})-2\lambda_{2}^{2}FE_{2}(b_{2})\\
&+\ 2E_{1}(\lambda_{2}^{2})\Gamma_{\ov 1 \ov 1}^{1}+2FE_{1}(\lambda_{2}^{2})\Gamma_{11}^{\ov  1}-
2E_{2}(\lambda_{1}^{2})\Gamma_{\ov 2 \ov 2}^{2}-2FE_{2}(\lambda_{1}^{2})\Gamma_{22}^{\ov 2}-\norm{\mu^{\HH}}V(\Lambda).
\end{align*}

The first terms of the first four lines will mainly give positive terms by virtue of the harmonicity condition
\begin{align*}
&- 2E_{1}(\lambda_{2}^{2})(\Gamma_{22}^{1}+\Gamma_{\ov 2 \ov 2}^{1})+
2E_{2}(\lambda_{1}^{2})(\Gamma_{11}^{2}+\Gamma_{\ov 1 \ov 1}^{2})-
2FE_{1}(\lambda_{2}^{2})(\Gamma_{22}^{\ov  1}+\Gamma_{\ov 2 \ov 2}^{\ov 1})\\
&+2FE_{2}(\lambda_{1}^{2})(\Gamma_{11}^{\ov 2}+\Gamma_{\ov 1 \ov 1}^{\ov  2})\\
&=2\Lambda(\Gamma_{22}^{1}+\Gamma_{\ov 2 \ov 2}^{1})^{2}+2\lambda_{1}^{2}a_{1}(\Gamma_{22}^{1}+\Gamma_{\ov 2 \ov 2}^{1})+
2\Lambda(\Gamma_{11}^{2}+\Gamma_{\ov 1 \ov 1}^{2})^{2}-2\lambda_{2}^{2}a_{2}(\Gamma_{11}^{2}+\Gamma_{\ov 1 \ov 1}^{2})\\
&+2\Lambda(\Gamma_{22}^{\ov  1}+\Gamma_{\ov 2 \ov 2}^{\ov 1})^{2}+
2\lambda_{1}^{2}b_{1}(\Gamma_{22}^{\ov  1}+\Gamma_{\ov 2 \ov 2}^{\ov 1})+
2\Lambda(\Gamma_{11}^{\ov 2}+\Gamma_{\ov 1 \ov 1}^{\ov  2})^{2}-2\lambda_{2}^{2}b_{2}(\Gamma_{11}^{\ov 2}+\Gamma_{\ov 1 \ov 1}^{\ov  2}).
\end{align*}

The second terms of the first four lines will produce curvature terms and some additional terms
\begin{align*}
&2\Lambda  \left( E_{1}(\Gamma_{22}^{1}+\Gamma_{\ov 2 \ov 2}^{1})+
E_{2}(\Gamma_{11}^{2}+\Gamma_{\ov 1 \ov 1}^{2})+
FE_{1}(\Gamma_{22}^{\ov  1}+\Gamma_{\ov 2 \ov 2}^{\ov 1})+
FE_{2}(\Gamma_{11}^{\ov 2}+\Gamma_{\ov 1 \ov 1}^{\ov 2}) \right)\\ 
&=2\Lambda \left((II_{a})+(II_{b})+(II_{c})
+R_{1212}+R_{1 \ov 2 1 \ov 2}+R_{\ov 1 2 \ov 1 2}+R_{\ov 1 \ov 2 \ov 1 \ov 2} \right),
\end{align*}
where
\begin{align*}
(II_{a})&=-g(\nabla_{E_{1}}E_{2}, \nabla_{E_{2}} E_{1})-g(\nabla_{E_{1}}FE_{2}, \nabla_{FE_{2}} E_{1})
-g(\nabla_{FE_{1}}E_{2}, \nabla_{E_{2}} FE_{1})\\
&-g(\nabla_{FE_{1}}FE_{2}, \nabla_{FE_{2}} FE_{1});\\
(II_{b})&=g(\nabla_{[E_{1}, E_{2}]}E_{2}, E_{1})+g(\nabla_{[E_{1}, FE_{2}]}FE_{2}, E_{1})+
g(\nabla_{[FE_{1}, E_{2}]}E_{2}, FE_{1})\\
&+g(\nabla_{[FE_{1}, FE_{2}]}FE_{2}, FE_{1});\\
(II_{c})&=g\left(\nabla_{E_{1}}E_{1}+\nabla_{FE_{1}}FE_{1}, \nabla_{E_{2}}E_{2}+\nabla_{FE_{2}}FE_{2}\right).
\end{align*}
Now, writing $\nabla_{E_{I}}E_{J}=\Gamma_{IJ}^{K}E_{K}$, one can verify that
\begin{align*}
&(II_{a})+(II_{b})=\sum_{I\neq K \neq 0}(\Gamma_{II}^{K})^{2}\\
&+2\sum_{i, K \neq 0}\Gamma_{i \ov \imath}^{K}\Gamma_{\ov \imath i}^{K}+
\sum_{\abs I =1 , \abs J =2}(\Gamma_{IJ}^{0}-\Gamma_{JI}^{0})\Gamma_{0J}^{I}-
\Gamma_{IJ}^{0}  \Gamma_{JI}^{0}.
\end{align*}
Most of the terms of $(II_{c})$ cancel with terms of the fifth line
\begin{align*}
(II_{c})=\sum_{\abs I =1 , \abs J =2 , I , J \neq K\neq 0}\Gamma_{II}^{K}\Gamma_{JJ}^{K} .
\end{align*}
Since
\begin{align*}
2E_{1}(\lambda_{2}^{2})\Gamma_{\ov 1 \ov 1}^{1}&=2\left[-\Lambda (\Gamma_{2 2}^{1}+\Gamma_{\ov 2 \ov 2}^{1})
-\lambda_{1}^{2} a_{1} \right]\Gamma_{\ov 1 \ov 1}^{1}\\
&=-2\Lambda (\Gamma_{2 2}^{1}+\Gamma_{\ov 2 \ov 2}^{1})\Gamma_{\ov 1 \ov 1}^{1}
-2\lambda_{1}^{2} a_{1} \Gamma_{\ov 1 \ov 1}^{1},
\end{align*}
and
\begin{align*}
-2E_{2}(\lambda_{1}^{2})\Gamma_{\ov 2 \ov 2}^{2}&=-2\left[\Lambda (\Gamma_{1 1}^{2}+\Gamma_{\ov 1 \ov 1}^{2})
-\lambda_{2}^{2} a_{2} \right]\Gamma_{\ov 2 \ov 2}^{2} \\
&=-2\Lambda (\Gamma_{1 1}^{2}+\Gamma_{\ov 1 \ov 1}^{2})\Gamma_{\ov 2 \ov 2}^{2}
+2\lambda_{2}^{2} a_{2} \Gamma_{\ov 2 \ov 2}^{2},
\end{align*}
we have
\begin{align}\label{2c}
&2\Lambda(II_{c})+\left[ \ 2E_{1}(\lambda_{2}^{2})\Gamma_{\ov 1 \ov 1}^{1}+2FE_{1}(\lambda_{2}^{2})\Gamma_{11}^{\ov  1}-
2E_{2}(\lambda_{1}^{2})\Gamma_{\ov 2 \ov 2}^{2}-2FE_{2}(\lambda_{1}^{2})\Gamma_{22}^{\ov 2}\right]\\
&=2\Lambda \sum_{\abs I =1 , \abs J =2 }\Gamma_{II}^{0}\Gamma_{JJ}^{0} \ -2\lambda_{1}^{2} (a_{1} \Gamma_{\ov 1 \ov 1}^{1}+
b_{1} \Gamma_{1 1}^{\ov 1})
+2\lambda_{2}^{2} (a_{2} \Gamma_{\ov 2 \ov 2}^{2}+b_{2} \Gamma_{2 2}^{\ov 2}).\notag
\end{align}

The third terms of each of the first four lines will cancel with the term $-d \Lambda(\nabla_{V}V)$ of $\Delta^{\VV}\Lambda$. 

Finally
\begin{align*}
-\norm{\mu^{\HH}}V(\Lambda)=&-2(\Gamma_{11}^{0}+\Gamma_{22}^{0})[V(\lambda_{1}^{2})-V(\lambda_{2}^{2})]\\
&=-2\left(\frac{V(\lambda_{1}^{2})}{2\lambda_{1}^{2}}+\frac{V(\lambda_{2}^{2})}{2\lambda_{2}^{2}} \right)
[V(\lambda_{1}^{2})-V(\lambda_{2}^{2})]\\
&=-\left(\frac{V(\lambda_{1}^{2})^{2}}{\lambda_{1}^{2}}-\frac{V(\lambda_{2}^{2})^{2}}{\lambda_{2}^{2}}+
V(\lambda_{1}^{2}) V(\lambda_{2}^{2}) (\lambda_{2}^{-2}-\lambda_{1}^{-2})\right).\\
\end{align*}

To make positive terms appear, we need a couple of lemmas.

\begin{lemma}
For a pseudo horizontally weakly conformal map
\begin{equation*}
\Gamma_{0 \ov \imath}^{j}+\Gamma_{0 i}^{\ov \jmath}+\Gamma_{\ov \imath j}^{0}+\Gamma_{i \ov \jmath}^{0}=0.
\end{equation*}
In particular, we have $\Gamma_{ii}^{0}=\Gamma_{\ov \imath \ov \imath}^{0}$ (i.e. $g(\nabla_{E_{i}}E_{i},V)=g(\nabla_{E_{\ov \imath}}E_{\ov \imath},V)$)
and $\Gamma_{i \ov \imath}^{0}=-\Gamma_{\ov \imath i}^{0}$.
\end{lemma}

\begin{proof}
As the induced $f$-structure $F$ is projectable, we have
\begin{align*}
0&=g([V, FE_{i}]-F[V, E_{i}], E_{j})\\
&=g(\nabla_{V}FE_{i}-\nabla_{FE_{i}}V -F\nabla_{V}E_{i}+F\nabla_{E_{i}}V, E_{j})\\
&=\Gamma_{0 \ov \imath}^{j}+\Gamma_{\ov \imath j}^{0}+\Gamma_{0 i}^{\ov \jmath}+\Gamma_{i \ov \jmath}^{0}.
\end{align*}
In particular, $\Gamma_{ii}^{0}=\Gamma_{\ov \imath \ov \imath}^{0}$, which could also be deduced from \eqref{nei}.
\end{proof}
Note that the first term in \eqref{2c} can then be rewritten
$$
2\Lambda \sum_{\abs I =1 , \abs J =2}\Gamma_{II}^{0}\Gamma_{JJ}^{0}=8\Lambda \Gamma_{11}^{0}\Gamma_{22}^{0}=
2V(\lambda_{1}^{2}) V(\lambda_{2}^{2}) (\lambda_{2}^{-2}-\lambda_{1}^{-2}).
$$

Next is a direct generalization of \cite[Lemma 2.3]{Baird}.

\begin{lemma}\label{lem4}
For any submersion, the quantities $\Gamma_{IJ}^{0}$ and $\Gamma_{0I}^{J}$ are related by
\begin{equation}\label{zero}
\Gamma_{IJ}^{0}=\frac{1}{\lambda_{\abs J}^{2}} h \left(\nabla_{V}^{\varphi}d \varphi (E_{I}),  d \varphi (E_{J}) \right)
-\Gamma_{0I}^{J}, \quad \forall I \neq J \neq 0.
\end{equation}
Therefore, they must satisfy the relation
\begin{equation}
(\lambda_{\abs I}^{2}-\lambda_{\abs J}^{2})\Gamma_{0I}^{J}=
\lambda_{\abs I}^{2} \Gamma_{JI}^{0} +\lambda_{\abs J}^{2} \Gamma_{IJ}^{0}, \quad \forall I \neq J \neq 0.
\end{equation}
\end{lemma}

\begin{proof}
Let $I \neq J \neq 0$, we have
\begin{align*}
g(\nabla_{E_{I}}E_{J},V)=&-g(E_{J},\nabla_{E_{I}}V)\\
&=-g(E_{J},[E_{I},V]+\nabla_{V}E_{I})\\
&=g(E_{J},[V,E_{I}])-g(\nabla_{V}E_{I},E_{J})\\
&=\frac{1}{\lambda_{\abs J}^{2}}h(d \varphi (E_{J}), \ d \varphi ([V,E_{I}]))-g(\nabla_{V}E_{I},E_{J})\\
&=\frac{1}{\lambda_{\abs J}^{2}}h(d \varphi (E_{J}), \ \nabla_{V}^{\varphi}d \varphi (E_{I}))-g(\nabla_{V}E_{I},E_{J}).
\end{align*}
Therefore
$$
\lambda_{\abs J}^{2} \left( \Gamma_{IJ}^{0}+ \Gamma_{0I}^{J} \right)=
h(d \varphi (E_{J}), \ \nabla_{V}^{\varphi}d \varphi (E_{I})),
$$
and using the skew-symmetry of the right-hand term, we get
$$
\lambda_{\abs J}^{2} \left( \Gamma_{IJ}^{0}+ \Gamma_{0I}^{J} \right)=
-\lambda_{\abs I}^{2} \left( \Gamma_{JI}^{0}+ \Gamma_{0J}^{I} \right).
$$
This completes the proof.

\end{proof}

Combining these lemmas, we obtain (cf. also \cite{Slobodeanu2}).
\begin{lemma}
For any pseudo horizontally weakly conformal submersion, the second fundamental form of the horizontal distribution is $F$-invariant
\begin{equation}
\Gamma_{\ov \imath j}^{0}+\Gamma_{j \ov \imath}^{0}=-(\Gamma_{i \ov \jmath}^{0}+\Gamma_{\ov \jmath i}^{0}).
\end{equation}
In particular $\Gamma_{\ov \imath \ov \jmath}^{0}+\Gamma_{\ov \jmath \ov \imath}^{0}=\Gamma_{i j}^{0}+\Gamma_{j i}^{0}$.
\end{lemma}

We conclude that, in the expression of $\Delta^{\HH} \Lambda$, the terms that contain $\Gamma_{IJ}^{0}$ or $\Gamma_{0I}^{J}$ are

\begin{align*}
&\lambda_{1}^{2}\left( -2(\Gamma_{1 \ov 1}^{0})^{2}+
\sum_{\abs I = 2} \Gamma_{01}^{I}(\Gamma_{1I}^{0}+\Gamma_{I1}^{0})+\Gamma_{1I}^{0} \Gamma_{I1}^{0}+
\Gamma_{0 \ov 1}^{I}(\Gamma_{\ov 1 I}^{0}+\Gamma_{I \ov 1}^{0})+\Gamma_{\ov 1 I}^{0} \Gamma_{I \ov 1}^{0}\right)\\
&-\lambda_{2}^{2} \left( -2(\Gamma_{2 \ov 2}^{0})^{2}+
\sum_{\abs I = 1} \Gamma_{02}^{I}(\Gamma_{2I}^{0}+\Gamma_{I2}^{0})+\Gamma_{2I}^{0} \Gamma_{I2}^{0}+
\Gamma_{0 \ov 2}^{I}(\Gamma_{\ov 2 I}^{0}+\Gamma_{I \ov 2}^{0})+\Gamma_{\ov 2 I}^{0} \Gamma_{I \ov 2}^{0} \right)\\
&-2\Lambda\sum_{\abs I = 1, \abs J = 2}\Gamma_{0I}^{J} (\Gamma_{IJ}^{0}-\Gamma_{JI}^{0})+ \Gamma_{IJ}^{0} \Gamma_{JI}^{0}.
\end{align*}
Using the notation $\tilde \Gamma_{0I}^{J}=h \left(\nabla_{V}^{\varphi}d \varphi (E_{I}),  d \varphi (E_{J}) \right)$ and Equation~\eqref{zero}, we obtain
\begin{align*}
\Gamma_{0I}^{J} (\Gamma_{IJ}^{0}+\Gamma_{JI}^{0})+ \Gamma_{IJ}^{0} \Gamma_{JI}^{0}=
-(\Gamma_{0I}^{J})^{2}
+2\frac{1}{\lambda_{\abs J}^{2}} \Gamma_{0I}^{J} \tilde \Gamma_{0I}^{J}
- \frac{1}{\lambda_{\abs I}^{2} \lambda_{\abs J}^{2}} (\tilde \Gamma_{0I}^{J})^{2},
\end{align*}
and 
\begin{align*}
\Gamma_{0I}^{J} (\Gamma_{IJ}^{0}-\Gamma_{JI}^{0})+ \Gamma_{IJ}^{0}  \Gamma_{JI}^{0}=
-3(\Gamma_{0I}^{J})^{2}
+2\left( \frac{1}{\lambda_{\abs I}^{2}}+\frac{1}{\lambda_{\abs J}^{2}} \right)\Gamma_{0I}^{J} \tilde \Gamma_{0I}^{J}
- \frac{1}{\lambda_{\abs I}^{2}  \lambda_{\abs J}^{2}} (\tilde \Gamma_{0I}^{J})^{2}.
\end{align*}
Therefore the terms in $\Gamma_{IJ}^{0}$ and $\Gamma_{0I}^{J}$ are
\begin{align*}
&-2\lambda_{1}^{2} (\Gamma_{1 \ov 1}^{0})^{2}+
2\lambda_{2}^{2} (\Gamma_{2 \ov 2}^{0})^{2}\\
&+\Lambda\sum_{\abs I = 1, \abs J = 2}
5(\Gamma_{0I}^{J})^{2}
-2 \left( \frac{1}{\lambda_{1}^{2}}+\frac{1}{\lambda_{2}^{2}} \right)\Gamma_{0I}^{J} \tilde \Gamma_{0I}^{J}
+ \frac{1}{\lambda_{1}^{2}  \lambda_{2}^{2}} (\tilde \Gamma_{0I}^{J})^{2}.
\end{align*}
Gathering the various terms yields the formula.
\end{proof}

\begin{proposition} 
Let $\varphi: (M^{5},g) \to (N^{4}, J, h)$ be an almost submersive pseudo harmonic morphism from a compact five-dimensional Riemannian manifold of strictly positive sectional curvature into a 4-dimensional almost Hermitian manifold such that $\grad |d\varphi|^2 \in \VV$. If the $f$-structure $F$ induced by $\varphi$ on $M$ is parallel on $\HH$ then $\varphi$ is a harmonic morphism.\\
If $N$ is $(1,2)$-symplectic, we can replace the condition of harmonicity by the minimality of the fibres.
\end{proposition}
\begin{proof}
Let $p$ be a maximum of the function $\Lambda$. Since the condition on $F$ forces $\varphi$ to be pseudo horizontally homothetic, it automatically has minimal fibres and 
$$
\Gamma_{i \ov \imath}^{K}\Gamma_{\ov \imath i}^{K}=-\Gamma_{i i}^{\ov K}\Gamma_{\ov \imath \ov \imath}^{\ov K}.
$$
Moreover, the condition
$$ \nabla_{E_I} FE_{J} = F\nabla_{E_I}E_J , \quad \forall I,J \neq 0,$$
implies that the symbols $\Gamma^{0}_{IJ}$ vanish, hence, if $I=1,\ov 1 , 2$ or $\ov 2$,
\begin{align*}
R_{0I0I} &= \langle R(E_I , V) V, E_I \rangle \\
&= \langle \nabla_{E_I}\nabla_{V}V - \nabla_{V}\nabla_{E_I}V - \nabla_{\nabla_{E_I}V}V + \nabla_{\nabla_{V}E_I}V ,E_I\rangle \\
&= -V(\Gamma^{I}_{I0}) - \Gamma^{K}_{I0}\Gamma^{I}_{0K} - \Gamma^{K}_{I0}\Gamma^{I}_{K0} + \Gamma^{K}_{0I}\Gamma^{I}_{K0}\\
&=0
\end{align*}
On the other hand, as $p$ is a critical point for $\Lambda$, $V(\lambda_{1}^{2})=V(\lambda_{2}^{2})$, so
$$\frac{V(\lambda_{1}^{2})^{2}}{2\lambda_{1}^{2}}-\frac{V(\lambda_{2}^{2})^{2}}{2\lambda_{2}^{2}}
+\Lambda  \frac{V(\lambda_{1}^{2})}{\lambda_{1}^{2}} \frac{V(\lambda_{2}^{2})}{\lambda_{2}^{2}} =
\frac{V(\lambda_{1}^{2})^{2}}{2}\left[\frac{1}{\lambda_{2}^{2}}-\frac{1}{\lambda_{1}^{2}}\right] .$$
As to the terms in $\Gamma_{IJ}^{0}$ and $\Gamma_{0I}^{J}$, it is easy to see that they must vanish.
Therefore the expression for $\Delta \Lambda$ at $p$ becomes
\begin{align*} 
\Delta \Lambda=& 4\Lambda \sum_{|I| \neq |K| \neq 0}(\Gamma_{II}^{K})^{2} 
+2\Lambda \left(R_{1212}+R_{1 \ov 2 1 \ov 2}+R_{\ov 1 2 \ov 1 2}+R_{\ov 1 \ov 2 \ov 1 \ov 2} \right)\\
&+\frac{V(\lambda_{1}^{2})^{2}}{2}\left[\frac{1}{\lambda_{2}^{2}}-\frac{1}{\lambda_{1}^{2}}\right] .
\end{align*}
Given our sign convention for the Laplacian, at a maximum point this expression should be negative, hence $\Lambda \equiv 0$.
\end{proof}

\begin{proposition}
Let $\varphi: (M^{5},g) \to (N^{4}, J, h)$ be an almost submersive pseudo harmonic morphism from a compact five-dimensional Riemannian manifold of constant strictly positive sectional curvature into a 4-dimensional almost Hermitian manifold such that $\grad |d\varphi|^2 \in \VV$. 
If the dilatation of $\varphi$, $\ell ^{2}=\frac{\lambda_{1}^{2}}{\lambda_{2}^{2}}$, is less than $\frac{3 + \sqrt 5}{2}$ and
the $f$-structure $F$ induced by $\varphi$ on $M$ satisfies
$$\left(\nabla_X F\right)X=0, \quad \forall X,$$
then $\varphi$ is a harmonic morphism.\\
If $N$ is $(1,2)$-symplectic, we can drop the condition of harmonicity.
\end{proposition}
\begin{proof}
The hypothesis on $F$ implies that $\Gamma_{1 \ov 1}^{0}=\Gamma_{2 \ov 2}^{0}=0$ and 
$$
\Gamma_{i \ov \imath}^{K}\Gamma_{\ov \imath i}^{K}=-\Gamma_{i i}^{\ov K}\Gamma_{\ov \imath \ov \imath}^{\ov K}.
$$
Moreover, it forces the fibres to be minimal.\\
The condition on the dilatation ensures that the polynomial in $\Gamma_{0I}^{J}$ and $\tilde \Gamma_{0I}^{J}$ is positive.
\end{proof}

When the domain is orientable, we can introduce contact geometry (cf.~\cite{Blair}).
\begin{theorem}
Let $(M^5,\phi, \xi, \eta, g)$ be a compact nearly cosymplectic almost contact metric manifold and $\varphi: M \to (N^4,J,h)$ a $(\phi, J)$-holomorphic almost submersion into a $(1,2)$-symplectic four-manifold. If the domain $M$ has constant strictly positive sectional curvature and the map $\varphi$ has $\grad |d\varphi|^2 \in \VV$, then $\varphi$ is horizontally weakly conformal and therefore a harmonic morphism.
\end{theorem}

\begin{proof}
According to \cite{Blair}, the condition of nearly symplectic is
$$(\nabla_{X}\phi)X=0, \forall X,$$
so, as in the previous Proposition, we only have to control the sign of the polynomial in $\Gamma_{0I}^{J}$ and $\tilde \Gamma_{0I}^{J}$.
From Lemma~\ref{lem4}, we deduce that 
\begin{align*}
\tilde \Gamma_{0I}^{J} =& h \left(\nabla_{V}^{\varphi}d \varphi (E_{I}),  d \varphi (E_{J}) \right)\\
&=\lambda_{\abs J}^{2} (\Gamma_{IJ}^{0}+\Gamma_{0I}^{J})\\
&=\lambda_{\abs J}^{2} \left(\Gamma_{IJ}^{0}+\frac{\lambda_{\abs I}^{2} \Gamma_{JI}^{0} + \lambda_{\abs J}^{2} \Gamma_{IJ}^{0}}{\lambda_{\abs I}^{2}-\lambda_{\abs J}^{2}}\right)\\
&=\frac{\lambda_{\abs I}^{2}\lambda_{\abs J}^{2} \left(\Gamma_{IJ}^{0}+ \Gamma_{JI}^{0}\right)}{\lambda_{\abs I}^{2}-\lambda_{\abs J}^{2}}.
\end{align*}
So if $\VV$ is a conformal foliation (equivalently $\HH$ is umbilic, cf. \cite[Proposition 2.5.8]{BW}), then $\tilde \Gamma_{0I}^{J}=0, \forall I \neq J$.
But on a nearly cosymplectic manifold $\xi$ is a Killing vector field \cite[Proposition 6.1]{Blair}, so in this case $\tilde \Gamma_{0I}^{J}=0, \forall I \neq J$ and
the polynomial in $\Gamma_{0I}^{J}$ and $\tilde \Gamma_{0I}^{J}$ must be positive.
\end{proof}

If we regard $\mathbb{S}^5$ as a totally geodesic hypersurface of $\mathbb{S}^6$ with its nearly K\"ahler structure, then the induced almost contact structure on $\mathbb{S}^5$ is nearly cosymplectic \cite[Example 4.5.3]{Blair}. Therefore

\begin{corollary}
A $(\phi, J)$-holomorphic almost submersion from $(\mathbb{S}^5,\phi, \xi, \eta, g)$ into a $(1,2)$-symplectic four-manifold with $\grad |d\varphi|^2 \in \VV$, must be horizontally weakly conformal and therefore a harmonic morphism.
\end{corollary}

\end{document}